\documentclass[11pt,a4paper]{article}
\usepackage{amsmath}
\usepackage{amsfonts}
\usepackage{amssymb}

\providecommand{\U}[1]{\protect \rule{.1in}{.1in}}
\newtheorem{theorem}{Theorem}

\newtheorem{corollary}[theorem]{Corollary}

\newtheorem{example}{Example}

\newtheorem{proposition}[theorem]{Proposition}

\newenvironment{proof}[1][Proof]{\noindent \textbf{#1.} }{\  \rule{0.5em}{0.5em}}
\include {mak}
\input{tcilatex}
\begin{document}

\begin{center}
{\Large An application of the partial r-Bell polynomials }

{\Large on some family of bivariate polynomials }

\  \  \  \ 

{\large Miloud Mihoubi\footnotemark[1] \ and \ Yamina Saidi\footnotemark[2]}

USTHB, Faculty of Mathematics, RECITS Laboratory, PB 32 El Alia 16111
Algiers, Algeria.

{\large \footnotemark[1]}mmihoubi@usthb.dz \ {\large \footnotemark[1]}%
miloudmihoubi@gmail.com \ {\large \footnotemark[2]}y\_saidi34@yahoo.com

{\large \  \  \ }
\end{center}

\noindent \textbf{Abstract. }The aim of this paper is to give some
combinatorial relations linked polynomials generalizing those of Appell type
to the partial $r$-Bell polynomials. We give an inverse relation, recurrence
relations involving some family of polynomials and their exact expressions
at rational values in terms of the partial $r$-Bell polynomials. We
illustrate the obtained results by various comprehensive examples.



\section{Introduction}

This work is motivated by the work of Mihoubi and Tiachachat \cite{mih1} on
the expressions of the Bernoulli polynomials at rational numbers by the
whitney numbers and the work of Mihoubi and Saidi \cite{mih3} on the
polynomials of Appell type. \newline
The aim of this paper is to give some combinatorial relations linked
polynomials generalizing those of Appell type to the partial $r$-Bell
polynomials. For given two sequences of real numbers $\mathbf{a=}\left(
a_{1},a_{2},\ldots \right) $ and $\mathbf{b=}\left( b_{1},b_{2},\ldots
\right) ,$ recall that the partial $r$-Bell polynomials%
\begin{equation*}
B_{n,k}^{\left( r\right) }\left( \mathbf{a};\mathbf{b}\right)
:=B_{n,k}^{\left( r\right) }\left( \left( a_{j}\right) ;\left( b_{j}\right)
\right) =B_{n,k}^{\left( r\right) }\left( a_{1},a_{2},\ldots
;b_{1},b_{2},\ldots \right)
\end{equation*}%
are defined by their generating function to be%
\begin{equation*}
\underset{n\geq k}{\sum }B_{n+r,k+r}^{\left( r\right) }\left( \mathbf{a};%
\mathbf{b}\right) \frac{t^{n}}{n!}=\frac{1}{k!}\left( \underset{j\geq 1}{%
\sum }a_{j}\frac{t^{j}}{j!}\right) ^{k}\left( \underset{j\geq 0}{\sum }%
b_{j+1}\frac{t^{j}}{j!}\right) ^{r}.
\end{equation*}%
These polynomials present a naturel extension of the partial Bell
polynomials \cite{bell} and generalize the $r$-Whitney numbers of both
kinds, the $r$-Lah numbers and the $r$-Whitney-Lah numbers. Mihoubi et
Rahmani \cite{mih4} introduced and studied these polynomials for which they
gave combinatorial and probabilistic interpretations and several properties.
Shattuck \cite{sha} gave more properties and Chouria and Luque \cite{chou}
defined three versions of the partial $r$-Bell polynomials in three
combinatorial Hopf algebras. For an application of these polynomials on a
family of bivariate polynomials, let us define this family. Indeed, let $A,\
B$ and $H$ be three analytic functions around zero with $A\left( 0\right)
=0, $ $A^{\prime }\left( 0\right) =B\left( 0\right) =1$ and let $\alpha $
and $x$ be real numbers. A sequence of numbers $P_{n}^{\left( \alpha \right)
}\left( A,H\right) $ is defined by%
\begin{equation}
\underset{n\geq 0}{\sum }P_{n}^{\left( \alpha \right) }\left( A,H\right) 
\frac{t^{n}}{n!}=\left( \frac{t}{A\left( t\right) }\right) ^{\alpha }H\left(
t\right)  \label{g}
\end{equation}%
and a sequence of polynomials $P_{n}^{\left( \alpha \right) }\left( x,y\mid
A,B,H\right) $ is to be%
\begin{equation}
\underset{n\geq 0}{\sum }P_{n}^{\left( \alpha \right) }\left( x,y\mid
A,B,H\right) \frac{t^{n}}{n!}=\left( \frac{t}{A\left( t\right) }\right)
^{\alpha }\left( A^{\prime }\left( t\right) \right) ^{x}\left( B\left(
t\right) \right) ^{y}H\left( t\right) .  \label{a}
\end{equation}%
In the second section we give an inverse relation linked to the partial $r$%
-Bell polynomials and in the third section we give recurrence relations and
exact expressions at rational values for the bivariate polynomials defined
above.

\section{The partial $r$-Bell polynomials and inverse relations}

For any power series $A\left( t\right) =\underset{j\geq 1}{\sum }a_{j}\frac{%
t^{j}}{j!}$ with $a_{1}\neq 0,$ below, we denote by $\overline{A}\left(
t\right) =\underset{j\geq 1}{\sum }\overline{a}_{j}\frac{t^{j}}{j!}$ for the
compositional inverse of $A\left( t\right) \ $and we let $\mathbf{a:=}\left(
a_{1},a_{2},\ldots \right) $ and $\overline{\mathbf{a}}:\mathbf{=}\left( 
\overline{a}_{1},\overline{a}_{2},\ldots \right) .$

The following theorem gives inverse relations linked to the partial $r$-Bell
polynomials.

\begin{theorem}
\label{T1}The following inverse relations hold%
\begin{equation*}
U_{n}=\underset{k=0}{\overset{n}{\sum }}B_{n+r,k+r}^{\left( r\right) }\left( 
\mathbf{a;b}\right) V_{k},\  \  \  \ V_{n}=\underset{k=0}{\overset{n}{\sum }}%
B_{n+r,k+r}^{\left( r\right) }\left( \overline{\mathbf{a}}\mathbf{;}\left( 
\mathbf{b\circ }\overline{\mathbf{a}}\right) ^{-1}\right) U_{k},
\end{equation*}%
i.e.%
\begin{equation*}
\underset{k=j}{\overset{n}{\sum }}B_{n+r,k+r}^{\left( r\right) }\left( 
\mathbf{a;b}\right) B_{k+r,j+r}^{\left( r\right) }\left( \overline{\mathbf{a}%
}\mathbf{;}\left( \mathbf{b\circ }\overline{\mathbf{a}}\right) ^{-1}\right)
=\delta _{\left( j,n\right) },
\end{equation*}%
where $\mathbf{b\circ }\overline{\mathbf{a}}$ is the sequence of the
coefficients of the power series $B\left( \overline{A}\left( t\right)
\right) $ and $\left( \mathbf{b\circ }\overline{\mathbf{a}}\right) ^{-1}$ is
the sequence of the coefficients of the power series $\left( B\left( 
\overline{A}\left( t\right) \right) \right) ^{-1}.$
\end{theorem}

\begin{proof}
If $U\left( t\right) =\underset{n\geq 0}{\sum }U_{n}\frac{t^{n}}{n!}$ and $%
V\left( t\right) =\underset{n\geq 0}{\sum }V_{n}\frac{t^{n}}{n!},$ then $%
U\left( t\right) =V\left( A\left( t\right) \right) \left( B\left( t\right)
\right) ^{r},$ or equivalently, by replacing $t$ by $\overline{A}\left(
t\right) ,$ it becomes $V\left( t\right) =U\left( \overline{A}\left(
t\right) \right) \left( B\left( \overline{A}\left( t\right) \right) \right)
^{-r}.$
\end{proof}

\begin{example}
For $A\left( t\right) =\frac{\exp \left( mt\right) -1}{m}$ and $B\left(
t\right) =\exp \left( t\right) $ we get 
\begin{eqnarray*}
\overline{A}\left( t\right) &=&\frac{\ln \left( 1+mt\right) }{m}, \\
\left( B\left( \overline{A}\left( t\right) \right) \right) ^{-1} &=&\left(
1+mt\right) ^{-\frac{1}{m}}, \\
B_{n+r,k+r}^{\left( r\right) }\left( \mathbf{a};\mathbf{b}\right)
&=&W_{m,r}\left( n,k\right) , \\
B_{n+r,k+r}^{\left( r\right) }\left( \overline{\mathbf{a}};\left( \mathbf{%
b\circ }\overline{\mathbf{a}}\right) ^{-1}\right) &=&w_{m,r}\left(
n,k\right) .
\end{eqnarray*}%
So, we obtain the inverse relations%
\begin{equation*}
U_{n}=\underset{k=0}{\overset{n}{\sum }}W_{m,r}\left( n,k\right) V_{k},\  \  \
\ V_{n}=\underset{k=0}{\overset{n}{\sum }}w_{m,r}\left( n,k\right) U_{k},
\end{equation*}%
where $w_{m,r}\left( n,k\right) $ and $W_{m,r}\left( n,k\right) $ are,
respectively, the whitney numbers of the first and second kind, for more
information, see for example \cite{bel,man,mer,mez,mih4,rah}.
\end{example}

\noindent For $B\left( t\right) =A^{\prime }\left( t\right) $ in Theorem \ref%
{T1} we obtain:

\begin{corollary}
\label{C1}The following inverse relations hold%
\begin{equation*}
U_{n}=\underset{k=0}{\overset{n}{\sum }}B_{n+r,k+r}^{\left( r\right) }\left( 
\mathbf{a}\right) V_{k},\  \  \  \ V_{n}=\underset{k=0}{\overset{n}{\sum }}%
B_{n+r,k+r}^{\left( r\right) }\left( \overline{\mathbf{a}}\right) U_{k},
\end{equation*}%
where $B_{n+r,k+r}^{\left( r\right) }\left( \mathbf{a}\right)
:=B_{n+r,k+r}^{\left( r\right) }\left( \mathbf{a;a}\right) .$
\end{corollary}

\noindent For $A\left( t\right) =\left( 1-t\right) ^{\beta }-1$ and $B\left(
t\right) =\left( 1-t\right) ^{\alpha }$ in Theorem \ref{T1} we obtain:

\begin{corollary}
Let $\alpha ,\beta $ be two real numbers such that $\beta \neq 0.$ \newline
Then, the following inverse relations hold%
\begin{eqnarray*}
U_{n} &=&\underset{k=0}{\overset{n}{\sum }}B_{n+r,k+r}^{\left( r\right)
}\left( \left( \beta \right) _{j};\left( \alpha \right) _{j-1}\right) V_{k},
\\
V_{n} &=&\underset{k=0}{\overset{n}{\sum }}\left( -1\right)
^{n-k}B_{n+r,k+r}^{\left( r\right) }\left( \left( -1/\beta \right) _{j}%
\mathbf{;}\left \langle \alpha /\beta \right \rangle _{j-1}\right) U_{k},
\end{eqnarray*}%
where $\left( \alpha \right) _{n}:=\alpha \left( \alpha -1\right) \cdots
\left( \alpha -n+1\right) \ $if$\ n\geq 1,\  \left( \alpha \right) _{0}:=1$
and $\left \langle \alpha \right \rangle _{n}:=\left( -1\right) ^{n}\left(
-\alpha \right) _{n}.$
\end{corollary}

\noindent For $B\left( t\right) =A\left( t\right) +1$ in Theorem \ref{T1} we
obtain:

\begin{corollary}
Then, the following inverse relations hold%
\begin{eqnarray*}
U_{n} &=&\underset{k=0}{\overset{n}{\sum }}B_{n+r,k+r}^{\left( r\right)
}\left( \left( a_{j}\right) ;\left( a_{j-1}\right) \right) V_{k},\  \  \
a_{0}=1, \\
V_{n} &=&\underset{k=0}{\overset{n}{\sum }}B_{n+r,k+r}^{\left( r\right)
}\left( \left( \overline{a}_{j}\right) ;\left( \left( -1\right) ^{j-1}\left(
j-1\right) !\right) \right) U_{k}.
\end{eqnarray*}
\end{corollary}

\section{Recurrence relations involving the polynomials $P_{n}^{\left( 
\protect \alpha \right) }$}

In \cite{mih3} we are proved the following identity%
\begin{equation}
P_{n}^{\left( \alpha \right) }\left( A,H\right) =P_{n}^{\left( n+1-\alpha
\right) }\left( B,\left( A^{\prime }\circ B\right) ^{-1}\left( H\circ
B\right) \right) ,  \label{s}
\end{equation}%
from which we gave%
\begin{equation*}
P_{n}^{\left( n+1+\alpha \right) }\left( A,H\right) =\underset{k=0}{\overset{%
n}{\sum }}B_{n,k}\left( \overline{\mathbf{a}}\right) P_{k}^{\left( \alpha
\right) }\left( A,\left( A^{\prime }\right) ^{-1}H\right) ,
\end{equation*}%
and, in particular, if we let $\left( \frac{t}{A\left( t\right) }\right)
^{\alpha }\left( A^{\prime }\left( t\right) \right) ^{x}=\underset{n\geq 0}{%
\sum }T_{n}^{\left( \alpha \right) }\left( x\mid A\right) \frac{t^{n}}{n!},$
we get%
\begin{equation*}
T_{n}^{\left( n+1+\alpha \right) }\left( x+1\mid A\right) =\underset{k=0}{%
\overset{n}{\sum }}B_{n,k}\left( \overline{\mathbf{a}}\right) T_{k}^{\left(
\alpha \right) }\left( x\mid A\right) .
\end{equation*}%
The following theorem generalizes this last identity as follows:

\begin{theorem}
\label{T2}There holds%
\begin{equation}
P_{n}^{\left( \alpha \right) }\left( x,y\mid A,B,H\right) =\underset{k=0}{%
\overset{n}{\sum }}B_{n+r,k+r}^{\left( r\right) }\left( \mathbf{a;b}\right)
P_{k}^{\left( k+1+\alpha \right) }\left( x+1,y-r\mid A,B,H\right) ,
\label{e1}
\end{equation}%
or equivalently%
\begin{equation}
P_{n}^{\left( n+1+\alpha \right) }\left( x,y\mid A,B,H\right) =\underset{k=0}%
{\overset{n}{\sum }}B_{n+r,k+r}^{\left( r\right) }\left( \overline{\mathbf{a}%
};\left( \mathbf{b\circ }\overline{\mathbf{a}}\right) ^{-1}\right)
P_{k}^{\left( \alpha \right) }\left( x-1,y+r\mid A,B,H\right) .  \label{f1}
\end{equation}
\end{theorem}

\begin{proof}
We prove that the two sides of (\ref{f1}) have the same exponential
generating functions. Indeed,%
\begin{eqnarray*}
&&\underset{n\geq 0}{\sum }\left( \underset{k=0}{\overset{n}{\sum }}%
B_{n+r,k+r}^{\left( r\right) }\left( \mathbf{a;b}\right) P_{k}^{\left(
k+1+\alpha \right) }\left( x,y\mid A,B,H\right) \right) \frac{t^{n}}{n!} \\
&=&\underset{k\geq 0}{\sum }P_{k}^{\left( k+1+\alpha \right) }\left( x,y\mid
A,B,H\right) \left( \underset{n\geq k}{\sum }B_{n+r,k+r}^{\left( r\right)
}\left( \mathbf{a;b}\right) \frac{t^{n}}{n!}\right) \\
&=&\underset{k\geq 0}{\sum }P_{k}^{\left( k+1+\alpha \right) }\left( x,y\mid
A,B,H\right) \frac{\left( A\left( t\right) \right) ^{k}}{k!}\left( B\left(
t\right) \right) ^{r}
\end{eqnarray*}%
and, by identity (\ref{s}), the following identity%
\begin{eqnarray*}
P_{k}^{\left( k+1+\alpha \right) }\left( x,y\mid A,B,H\right) &=&D_{t=0}^{k} 
\left[ \left( \frac{t}{A\left( t\right) }\right) ^{k+1+\alpha }\left(
A^{\prime }\left( t\right) \right) ^{x}\left( B\left( t\right) \right)
^{y}H\left( t\right) \right] \\
&=&D_{t=0}^{k}\left[ \left( \frac{t}{\overline{A}\left( t\right) }\right)
^{-\alpha }\left( A^{\prime }\circ \overline{A}\left( t\right) \right)
^{x-1}\left( B\circ \overline{A}\left( t\right) \right) ^{y}\left( H\circ 
\overline{A}\left( t\right) \right) \right] \\
&=&D_{t=0}^{k}\left[ \left( \frac{t}{\overline{A}\left( t\right) }\right)
^{-\alpha }\left( \overline{A}^{\prime }\left( t\right) \right) ^{1-x}\left(
B\circ \overline{A}\left( t\right) \right) ^{y}\left( H\circ \overline{A}%
\left( t\right) \right) \right] \\
&=&P_{k}^{\left( -\alpha \right) }\left( 1-x,y\mid \overline{A},B\circ 
\overline{A},H\circ \overline{A}\right) ,
\end{eqnarray*}%
shows that the last expansion becomes%
\begin{eqnarray*}
&&\underset{k\geq 0}{\sum }P_{k}^{\left( -\alpha \right) }\left( 1-x,y\mid 
\overline{A},B\circ \overline{A},H\circ \overline{A}\right) \frac{\left(
A\left( t\right) \right) ^{k}}{k!}\left( B\left( t\right) \right) ^{r} \\
&=&\left( \frac{A\left( t\right) }{\overline{A}(A\left( t\right) )}\right)
^{-\alpha }\left( \overline{A}^{\prime }\left( A\left( t\right) \right)
\right) ^{1-x}\left( B\left( \overline{A}\left( A\left( t\right) \right)
\right) \right) ^{y}H\left( \overline{A}\left( A\left( t\right) \right)
\right) \left( B\left( t\right) \right) ^{r} \\
&=&\left( \frac{t}{A\left( t\right) )}\right) ^{\alpha }\left( A\left(
t\right) \right) ^{x-1}\left( B\left( t\right) \right) ^{y+r}H\left( t\right)
\\
&=&\underset{n\geq 0}{\sum }P_{n}^{\left( \alpha \right) }\left( x-1,y+r\mid
A,B,H\right) \frac{t^{n}}{n!}.
\end{eqnarray*}%
The first identity of this theorem can be obtained from the second one by
apply Theorem \ref{T1}.
\end{proof}

\noindent Let $\left( L_{n}^{\left( \alpha ,\beta \right) }\left( x\right)
;n\geq 0\right) $ be a sequence of polynomials defined by%
\begin{equation*}
\underset{n\geq 0}{\sum }L_{n}^{\left( \alpha ,\beta \right) }\left(
x\right) t^{n}=\left( 1-t\right) ^{\alpha }\exp \left( x\left( \left(
1-t\right) ^{\beta }-1\right) \right) .
\end{equation*}%
Then, the application of Theorem \ref{T2} to the sequence $\left(
L_{n}^{\left( \alpha ,\beta \right) }\left( x\right) ;n\geq 0\right) $ gives:

\begin{corollary}
There holds%
\begin{equation*}
L_{n}^{\left( \alpha ,\beta \right) }\left( x+y\right) =\underset{k=0}{%
\overset{n}{\sum }}\left( -1\right) ^{k}L_{n-k}^{\left( k,\beta \right)
}\left( x\right) \left( L_{k}^{\left( -\alpha -2,-\beta \right) }\left(
y\right) -2kL_{k-1}^{\left( -\alpha -3,-\beta \right) }\left( y\right)
\right) .
\end{equation*}
\end{corollary}

\begin{proof}
For $A\left( t\right) =t-t^{2},$ $B\left( t\right) =\exp \left( \left(
1-t\right) ^{\beta }-1\right) $ and $H\left( t\right) =1$ we get%
\begin{eqnarray*}
\underset{n\geq 0}{\sum }B_{n+r,k+r}^{\left( r\right) }\left( \mathbf{a;b}%
\right) \frac{t^{n}}{n!} &=&\frac{t^{k}}{k!}\left( 1-t\right) ^{k}\exp
\left( r\left( \left( 1-t\right) ^{\beta }-1\right) \right)  \\
&=&\frac{1}{k!}\underset{n\geq k}{\sum }L_{n-k}^{\left( k,\beta \right)
}\left( r\right) t^{n}, \\
\underset{n\geq 0}{\sum }P_{n}^{\left( -\alpha \right) }\left( 1,y\mid
A,B,1\right) \frac{t^{n}}{n!} &=&\left( \frac{t}{t-t^{2}}\right) ^{-\alpha
}\left( 1-2t\right) \exp \left( y\left( \left( 1-t\right) ^{\beta }-1\right)
\right)  \\
&=&\underset{n\geq 0}{\sum }\left( L_{n}^{\left( \alpha ,\beta \right)
}\left( y\right) -2nL_{n-1}^{\left( \alpha ,\beta \right) }\left( y\right)
\right) t^{n}, \\
\underset{n\geq 0}{\sum }P_{n}^{\left( -\alpha \right) }\left( 0,y\mid
A,B,1\right) \frac{t^{n}}{n!} &=&\underset{n\geq 0}{\sum }L_{n}^{\left(
\alpha ,\beta \right) }\left( y\right) t^{n}.
\end{eqnarray*}%
So, we get%
\begin{eqnarray*}
B_{n+r,k+r}^{\left( r\right) }\left( \mathbf{a;b}\right)  &=&\frac{n!}{k!}%
L_{n-k}^{\left( k,\beta \right) }\left( r\right) , \\
P_{n}^{\left( -\alpha \right) }\left( 1,y\mid A,B,1\right)  &=&n!\left(
L_{n}^{\left( \alpha ,\beta \right) }\left( y\right) -2nL_{n-1}^{\left(
\alpha ,\beta \right) }\left( y\right) \right) , \\
P_{n}^{\left( -\alpha \right) }\left( 0,y\mid A,B,1\right) 
&=&n!L_{n}^{\left( \alpha ,\beta \right) }\left( y\right) .
\end{eqnarray*}%
Then, from (\ref{e1}) we get 
\begin{equation*}
L_{n}^{\left( \alpha ,\beta \right) }\left( y\right) =\underset{k=0}{\overset%
{n}{\sum }}L_{n-k}^{\left( k,\beta \right) }\left( r\right) \left(
L_{k}^{\left( k+1+\alpha ,\beta \right) }\left( y-r\right)
-2kL_{k-1}^{\left( k+1+\alpha ,\beta \right) }\left( y-r\right) \right) 
\end{equation*}%
and by the identity $L_{n}^{\left( \alpha ,\beta \right) }\left( x\right)
=\left( -1\right) ^{n}L_{n}^{\left( n-1-\alpha ,-\beta \right) }\left(
x\right) $ \cite{mih5} we get%
\begin{equation*}
L_{n}^{\left( \alpha ,\beta \right) }\left( r+y\right) =\underset{k=0}{%
\overset{n}{\sum }}\left( -1\right) ^{k}L_{n-k}^{\left( k,\beta \right)
}\left( r\right) \left( L_{k}^{\left( -\alpha -2,-\beta \right) }\left(
y\right) -2kL_{k-1}^{\left( -\alpha -3,-\beta \right) }\left( y\right)
\right) .
\end{equation*}%
Now, since the polynomial%
\begin{equation*}
Q_{n}\left( x\right) =L_{n}^{\left( \alpha ,\beta \right) }\left( x+y\right)
-\underset{k=0}{\overset{n}{\sum }}\left( -1\right) ^{k}L_{n-k}^{\left(
k,\beta \right) }\left( x\right) \left( L_{k}^{\left( -\alpha -2,-\beta
\right) }\left( y\right) -2kL_{k-1}^{\left( -\alpha -3,-\beta \right)
}\left( y\right) \right) 
\end{equation*}%
vanishes on each non-negative integer $r,$ the desired identity follows.
\end{proof}

\noindent Recall that the high order Bernoulli polynomials of the first kind 
$\left( B_{n}^{\left( \alpha \right) }\left( x\right) \right) $ are defined
by%
\begin{equation*}
\underset{n\geq 0}{\sum }B_{n}^{\left( \alpha \right) }\left( x\right) \frac{%
t^{n}}{n!}=\left( \frac{t}{\exp \left( t\right) -1}\right) ^{\alpha }\exp
\left( xt\right)
\end{equation*}%
with $B_{n}^{\left( 1\right) }\left( x\right) =B_{n}\left( x\right) $ is the 
$n$-th Bernoulli polynomials of the first kind, see \cite{luk,pra,rom,sri2}.

\begin{corollary}
There hold%
\begin{eqnarray}
B_{n}^{\left( \alpha \right) }\left( x\right) &=&\underset{k=0}{\overset{n}{%
\sum }}\frac{W_{m,r}\left( n,k\right) }{m^{n-k}}B_{k}^{\left( k+1+\alpha
\right) }\left( x+1-\frac{r}{m}\right) ,  \label{f} \\
B_{n}^{\left( n+1+\alpha \right) }\left( x\right) &=&\underset{k=0}{\overset{%
n}{\sum }}\frac{w_{m,r}\left( n,k\right) }{m^{n-k}}B_{k}^{\left( \alpha
\right) }\left( x-1+\frac{r}{m}\right)  \label{fbis}
\end{eqnarray}%
and%
\begin{eqnarray*}
B_{n}^{\left( n+p+1\right) }\left( 1-\frac{r}{m}\right) &=&\frac{1}{m^{n}}%
\binom{n+p}{p}^{-1}w_{m,r}\left( n+p,p\right) , \\
B_{n}^{\left( p\right) }\left( \frac{r-s}{m}\right) &=&\frac{1}{m^{n}}%
\underset{k=0}{\overset{n}{\sum }}\binom{k+p}{p}^{-1}W_{m,r}\left(
n,k\right) w_{m,s}\left( k+p,p\right) .
\end{eqnarray*}
\end{corollary}

\begin{proof}
For $A\left( t\right) =\frac{1}{m}\left( \exp \left( mt\right) -1\right) ,$ $%
B\left( t\right) =\exp \left( t\right) $ and $H\left( t\right) =1$ we get%
\begin{eqnarray*}
\overline{A}\left( t\right) &=&\frac{1}{m}\ln \left( 1+mt\right) , \\
\left( B\left( \overline{A}\left( t\right) \right) \right) ^{-1} &=&\left(
1+mt\right) ^{-\frac{1}{m}}, \\
B_{n+r,k+r}^{\left( r\right) }\left( \mathbf{a};\mathbf{b}\right)
&=&W_{m,r}\left( n,k\right) , \\
B_{n+r,k+r}^{\left( r\right) }\left( \overline{\mathbf{a}};\left( \mathbf{%
b\circ }\overline{\mathbf{a}}\right) ^{-1}\right) &=&w_{m,r}\left(
n,k\right) , \\
P_{n}^{\left( \alpha \right) }\left( x,y\mid A,B,1\right)
&=&m^{n}B_{n}^{\left( \alpha \right) }\left( x+\frac{y}{m}\right) .
\end{eqnarray*}%
Then, from Theorem \ref{T2} we get%
\begin{eqnarray*}
B_{n}^{\left( \alpha \right) }\left( x\right) &=&\underset{k=0}{\overset{n}{%
\sum }}\frac{W_{m,r}\left( n,k\right) }{m^{n-k}}B_{k}^{\left( k+1+\alpha
\right) }\left( x+1-\frac{r}{m}\right) , \\
B_{n}^{\left( n+1+\alpha \right) }\left( x\right) &=&\underset{k=0}{\overset{%
n}{\sum }}\frac{w_{m,r}\left( n,k\right) }{m^{n-k}}B_{k}^{\left( \alpha
\right) }\left( x-1+\frac{r}{m}\right) ,
\end{eqnarray*}%
and, since $\frac{d}{dx}B_{n}^{\left( \alpha \right) }\left( x\right)
=nB_{n-1}^{\left( \alpha \right) }\left( x\right) ,$ if one differenciate
the two sides of this last identity $p$ times he obtain%
\begin{equation*}
B_{n}^{\left( n+p+1+\alpha \right) }\left( x\right) =\binom{n+p}{p}^{-1}%
\underset{k=0}{\overset{n}{\sum }}\binom{k+p}{p}\frac{w_{m,r}\left(
n+p,k+p\right) }{m^{n-k}}B_{k}^{\left( \alpha \right) }\left( x-1+\frac{r}{m}%
\right) .
\end{equation*}%
Then, for $\alpha =0,$ there holds%
\begin{equation*}
B_{n}^{\left( n+p+1\right) }\left( x\right) =\binom{n+p}{p}^{-1}\underset{k=0%
}{\overset{n}{\sum }}\binom{k+p}{p}\frac{w_{m,r}\left( n+p,k+p\right) }{%
m^{n-k}}\left( x-1+\frac{r}{m}\right) ^{k}
\end{equation*}%
which gives for $x=1-\frac{r}{m}:$%
\begin{equation*}
B_{n}^{\left( n+p+1\right) }\left( 1-\frac{r}{m}\right) =\binom{n+p}{p}^{-1}%
\frac{w_{m,r}\left( n+p,p\right) }{m^{n}}.
\end{equation*}%
Upon using this identity, identity (\ref{f}) becomes when $\alpha =p:$%
\begin{equation*}
B_{n}^{\left( p\right) }\left( \frac{r-s}{m}\right) =\frac{1}{m^{n}}\underset%
{k=0}{\overset{n}{\sum }}\binom{k+p}{p}^{-1}W_{m,r}\left( n,k\right)
w_{m,s}\left( k+p,p\right) .
\end{equation*}
\end{proof}

\noindent The high order Bernoulli polynomials of the second kind $\left(
b_{n}^{\left( \alpha \right) }\left( x\right) \right) $ can be defined by%
\begin{equation*}
\underset{n\geq 0}{\sum }b_{n}^{\left( \alpha \right) }\left( x\right) \frac{%
t^{n}}{n!}=\left( \frac{t}{\ln \left( 1+t\right) }\right) ^{\alpha }\left(
1+t\right) ^{x}
\end{equation*}%
with $b_{n}^{\left( 1\right) }\left( x\right) =b_{n}\left( x\right) $ is the 
$n$-th Bernoulli polynomial of the second kind, see \cite{luk,pra,rom,sri2}.

\begin{corollary}
There hold%
\begin{eqnarray}
b_{n}^{\left( \alpha \right) }\left( x\right) &=&\underset{k=0}{\overset{n}{%
\sum }}\frac{w_{m,r}\left( n,k\right) }{m^{n-k}}b_{k}^{\left( k+1+\alpha
\right) }\left( x-1+\frac{r}{m}\right) ,  \label{h} \\
b_{n}^{\left( n+1+\alpha \right) }\left( x\right) &=&\underset{k=0}{\overset{%
n}{\sum }}\frac{W_{m,r}\left( n,k\right) }{m^{n-k}}b_{k}^{\left( \alpha
\right) }\left( x+1-\frac{r}{m}\right)  \label{hbis}
\end{eqnarray}%
and%
\begin{eqnarray*}
b_{n}^{\left( n+p+1\right) }\left( 1+\frac{r}{m}\right) &=&\frac{1}{m^{n}}%
\binom{n+p}{p}^{-1}W_{m,r}\left( n+p,p\right) , \\
b_{n}^{\left( p\right) }\left( \frac{s-r}{m}\right) &=&\frac{1}{m^{n}}%
\underset{k=0}{\overset{n}{\sum }}\binom{k+p}{p}^{-1}w_{m,r}\left(
n,k\right) W_{m,s}\left( k+p,p\right) .
\end{eqnarray*}
\end{corollary}

\begin{proof}
For $A\left( t\right) =\frac{1}{m}\ln \left( 1+mt\right) ,$ $B\left(
t\right) =\left( 1+mt\right) ^{-\frac{1}{m}}$ and $H\left( t\right) =1$ we
get%
\begin{eqnarray*}
\overline{A}\left( t\right) &=&\frac{1}{m}\left( e^{mt}-1\right) , \\
\left( B\left( \overline{A}\left( t\right) \right) \right) ^{-1} &=&e^{t}, \\
B_{n+r,j+r}^{\left( r\right) }\left( \mathbf{a;b}\right) &=&w_{m,r}\left(
n,k\right) , \\
B_{n+r,k+r}^{\left( r\right) }\left( \overline{\mathbf{a}};\left( \mathbf{%
b\circ }\overline{\mathbf{a}}\right) ^{-1}\right) &=&W_{m,r}\left(
n,k\right) , \\
P_{n}^{\left( \alpha \right) }\left( x,y\mid A,B,1\right)
&=&m^{n}b_{n}^{\left( \alpha \right) }\left( -x-\frac{y}{m}\right) .
\end{eqnarray*}%
Then, from Theorem \ref{T2} we get 
\begin{subequations}
\begin{eqnarray*}
b_{n}^{\left( \alpha \right) }\left( x\right) &=&\underset{k=0}{\overset{n}{%
\sum }}\frac{w_{m,r}\left( n,k\right) }{m^{n-k}}b_{k}^{\left( k+1+\alpha
\right) }\left( x-1+\frac{r}{m}\right) , \\
b_{n}^{\left( n+1+\alpha \right) }\left( x\right) &=&\underset{k=0}{\overset{%
n}{\sum }}\frac{W_{m,r}\left( n,k\right) }{m^{n-k}}b_{k}^{\left( \alpha
\right) }\left( x+1-\frac{r}{m}\right) ,
\end{eqnarray*}%
and, since $\frac{d}{dx}b_{n}^{\left( \alpha +1\right) }\left( x\right)
=nb_{n-1}^{\left( \alpha \right) }\left( x\right) ,$ if one differenciate
the two sides of this last identity $p$ times he obtain 
\end{subequations}
\begin{equation*}
b_{n}^{\left( n+p+1+\alpha \right) }\left( x\right) =\binom{n+p}{p}^{-1}%
\underset{k=0}{\overset{n}{\sum }}\binom{k+p}{p}\frac{W_{m,r}\left(
n+p,k+p\right) }{m^{n-k}}b_{k}^{\left( \alpha \right) }\left( x+1-\frac{r}{m}%
\right) .
\end{equation*}%
Then, for $\alpha =0,$ there holds%
\begin{equation*}
b_{n}^{\left( n+p+1\right) }\left( x\right) =\binom{n+p}{p}^{-1}\underset{k=0%
}{\overset{n}{\sum }}\binom{k+p}{p}\frac{W_{m,r}\left( n+p,k+p\right) }{%
m^{n-k}}\left( x+1-\frac{r}{m}\right) _{k}
\end{equation*}%
which gives for $x=-1+\frac{r}{m}:$%
\begin{equation*}
b_{n}^{\left( n+p+1\right) }\left( -1+\frac{r}{m}\right) =\binom{n+p}{p}^{-1}%
\frac{W_{m,r}\left( n+p,p\right) }{m^{n}}.
\end{equation*}%
Upon using this identity, identity (\ref{h}) becomes when $\alpha =p:$%
\begin{equation*}
b_{n}^{\left( p\right) }\left( \frac{s-r}{m}\right) =\frac{1}{m^{n}}\underset%
{k=0}{\overset{n}{\sum }}\binom{k+p}{p}^{-1}w_{m,r}\left( n,k\right)
W_{m,s}\left( k+p,p\right) .
\end{equation*}
\end{proof}

\noindent The above corollaries can be written in their general case as
follows.

\begin{proposition}
There holds%
\begin{equation*}
D_{t=0}^{n}\left[ \left( \frac{t}{A\left( t\right) }\right) ^{p}\left(
B\left( t\right) \right) ^{r-s}\right] =\frac{1}{p!}\underset{k=0}{\overset{n%
}{\sum }}\binom{k+p}{p}^{-1}B_{n+r,k+r}^{\left( r\right) }\left( \mathbf{a;b}%
\right) B_{k+p+s,p+s}^{\left( s\right) }\left( \overline{\mathbf{a}}\mathbf{;%
}\left( \mathbf{b\circ }\overline{\mathbf{a}}\right) ^{-1}\right) .
\end{equation*}
\end{proposition}

\begin{proof}
By definition and by Theorem \ref{T1}, we have%
\begin{eqnarray*}
D_{t=0}^{n}\left[ \left( \frac{t}{A\left( t\right) }\right) ^{p}\left(
B\left( t\right) \right) ^{r-s}\right] &=&P_{n}^{\left( p\right) }\left(
0,r-s\mid A,B,1\right) \\
&=&\underset{k=0}{\overset{n}{\sum }}B_{n+r,k+r}^{\left( r\right) }\left( 
\mathbf{a;b}\right) P_{k}^{\left( k+p+1\right) }\left( 1,-s\mid A,B,1\right)
,
\end{eqnarray*}%
and since%
\begin{eqnarray*}
P_{k}^{\left( k+p+1\right) }\left( 1,-s\mid A,B,1\right)
&=&D_{t=0}^{k}\left( \frac{t}{A\left( t\right) }\right) ^{k+p+1}A^{\prime
}\left( t\right) \left( B\left( t\right) \right) ^{-s} \\
&=&D_{t=0}^{k}\left( \frac{\overline{A}\left( t\right) }{t}\right)
^{p}\left( B\circ \overline{A}\left( t\right) \right) ^{-s} \\
&=&\frac{1}{p!}\binom{k+p}{p}^{-1}B_{k+p+s,p+s}^{\left( s\right) }\left( 
\overline{\mathbf{a}}\mathbf{;}\left( \mathbf{b\circ }\overline{\mathbf{a}}%
\right) ^{-1}\right) ,
\end{eqnarray*}%
it follows $P_{n}^{\left( p\right) }\left( 0,r-s\mid A,B,1\right) =\frac{1}{%
p!}\underset{k=0}{\overset{n}{\sum }}\binom{k+p}{p}^{-1}B_{n+r,k+r}^{\left(
r\right) }\left( \mathbf{a;b}\right) B_{k+p+s,p+s}^{\left( s\right) }\left( 
\mathbf{a;}\left( \mathbf{b\circ }\overline{\mathbf{a}}\right) ^{-1}\right)
. $
\end{proof}

\noindent Similarly, we also have:

\begin{proposition}
There holds%
\begin{equation*}
D_{t=0}^{n}\left[ \left( \frac{t}{A\left( t\right) }\right) ^{p}\left(
B\left( t\right) \right) ^{r}\left( A^{\prime }\left( t\right) \right) ^{-s}%
\right] =\frac{1}{p!}\underset{k=0}{\overset{n}{\sum }}\binom{k+p}{p}%
^{-1}B_{n+r,k+r}^{\left( r\right) }\left( \mathbf{a;b}\right)
B_{k+p+s,p+s}^{\left( s\right) }\left( \overline{\mathbf{a}}\right) .
\end{equation*}
\end{proposition}

\begin{proof}
By definition and by Theorem \ref{T1}, we have%
\begin{eqnarray*}
D_{t=0}^{n}\left[ \left( \frac{t}{A\left( t\right) }\right) ^{p}\left(
B\left( t\right) \right) ^{r}\left( A^{\prime }\left( t\right) \right) ^{-s}%
\right] &=&P_{n}^{\left( p\right) }\left( -s,r\mid A,B,1\right) \\
&=&\underset{k=0}{\overset{n}{\sum }}B_{n+r,k+r}^{\left( r\right) }\left( 
\mathbf{a;b}\right) P_{k}^{\left( k+p+1\right) }\left( 1-s,0\mid
A,B,1\right) ,
\end{eqnarray*}%
and since%
\begin{eqnarray*}
P_{k}^{\left( k+p+1\right) }\left( 1-s,0\mid A,B,1\right)
&=&D_{t=0}^{k}\left( \frac{t}{A\left( t\right) }\right) ^{k+p+1}\left(
A^{\prime }\left( t\right) \right) ^{1-s} \\
&=&D_{t=0}^{k}\left( \frac{\overline{A}\left( t\right) }{t}\right)
^{p}\left( A^{\prime }\circ \overline{A}\left( t\right) \right) ^{-s} \\
&=&D_{t=0}^{k}\left( \frac{\overline{A}\left( t\right) }{t}\right)
^{p}\left( \overline{A}^{\prime }\left( t\right) \right) ^{s} \\
&=&\frac{1}{p!}\binom{k+p}{p}^{-1}B_{k+p+s,p+s}^{\left( s\right) }\left( 
\overline{\mathbf{a}}\right) ,
\end{eqnarray*}%
it follows $P_{n}^{\left( p\right) }\left( -s,r\mid A,B,1\right) =\frac{1}{p!%
}\underset{k=0}{\overset{n}{\sum }}\binom{k+p}{p}^{-1}B_{n+r,k+r}^{\left(
r\right) }\left( \mathbf{a;b}\right) B_{k+p+s,p+s}^{\left( s\right) }\left( 
\overline{\mathbf{a}}\right) .$
\end{proof}

\end{document}